\documentclass[12pt]{amsart}
\usepackage{amsfonts, amsbsy, amsmath, amssymb}
\usepackage{todonotes}
\newtheorem{thm}{Theorem}[section]

\newtheorem{prop}[thm]{Proposition}
\newtheorem{exmp}[thm]{Example}

\newtheorem{rmk}[thm]{Remark}

\newtheorem{thm-con}[thm]{Theorem-Conjecture}
\numberwithin{equation}{section}

\theoremstyle{definition}
\newtheorem{defn}[thm]{Definition}

\newcommand{\f}{\Bbb F}

\begin{document}

\title[Alexander $f$-quandles on finite fields]{The Cocycle structure of the Alexander $f$-quandles on finite fields}

\author[Indu Rasika Churchill]{Indu Rasika Churchill}
\address{Department of Mathematics and Statitstics,
University of South Florida, Tampa, FL 33620}
\email{udyanganiesi@mail.usf.edu}

\author[Mohamed Elhamdadi]{Mohamed Elhamdadi}
\address{Department of Mathematics and Statitstics,
University of South Florida, Tampa, FL 33620}
\email{emohamed@usf.edu}

\author[Neranga Fernando]{Neranga Fernando}
\address{Department of Mathematics,
Northeastern University, Boston, MA 02115}
\email{w.fernando@northeastern.edu}

\begin{abstract} 
We determine the second, third, and fourth cohomology groups of Alexander $f$-quandles  of the form $\mathbb{F}_q[T,S]/ (T-\omega, S-\beta)$, where $\f_q$ denotes the finite field of order $q$, $\omega \in \f_q\setminus \{0,1\}$, and $\beta \in \f_q$.
\end{abstract}



\maketitle


\section{Introduction}
Quandles are in general non-associative structures whose axioms correspond to the algebraic distillation of the three Reidemeister moves in knot theory.  They were introduced independently in the 1980s by Joyce \cite{Joyce} and Matveev \cite{Matveev}.  Quandles  were used to construct representations of the braid groups.   
Thus giving constructions of invariants of knots and knotted
surfaces as can be seen in \cite{CEGS, CJKLS, CESY}. They have been also investigated in the topological context \cite{EM, R} and also for their own
right as other non-associative algebraic structures \cite{H1, H2,HSV, T}.  For more details and recent account on quandles see \cite{EN, N}.

Motivated by Hom-algebra structures \cite{MS}, $f$-racks,$f$-quandles and their cohomology theory were  introduced and investigated in \cite{CEGM-2016}. Explicit cocycles of this quandle cohomology may be used in the study of Knot Theory, thus in this paper,  we investigate the second, third, and fourth cohomology groups of Alexander $f$-quandles \cite{CEGM-2016}. Our work is  motivated by \cite{TM-2005}, \cite{TM-2003}, and \cite{TN-2013}.  Precisely we give basis for the cohomology group $H^n((X, *, f); \f_q)$ with $n=2,3$ and $4$.

   Through out this paper, let $p$ be a prime, $q = p^m$, and $ \mathbb{F}_q$  denote the finite field of order $q$. Let $M = \mathbb{Z}[\omega^{\pm},\beta] = \mathbb{F}_q$ , where  $\omega (\neq 1)$ and $\beta$ be non-zero elements of $\f_q$. Let $k$ be an algebraic closure of $\mathbb{F}_q$. 
For $n = 2, 3, 4$, we wish to calculate the Cohomology $H^n(\mathbb{F}_q[T,S]/(T-\omega, S-\beta),k)$ of the Alexander $f$-quandle $\mathbb{F}_q[T,S]/(T-\omega, S-\beta)$  with coefficients in $k$.

 At the end of each of sections 3, 4 and 5, we provide basis for $2$-cocycle, $3$-cocycle and $4$-cocycle in theorems 3.2, 4.10 and 5.9 respectively. The proofs of this theorems are similar to that of \cite{TM-2005}. These proofs will appear in  future work.

The paper is organized as follows. In Section 2, we present some preliminaries that will be used throughout the paper. In Sections ~\ref{2cocy}, ~\ref{3cocy}, and ~\ref{4cocy}, we survey $2$-cocycles, $3$-cocycles, and $4$-cocyles of Alexander $f$-quandles, respectively. We also give some examples in each section.


\section{Preliminaries}

In this section, we list some preliminaries that will be useful in latter sections. 

\begin{defn}(\cite[Definition 2.1]{CEGM-2016})
An $f$-quandle  is a set $X$ equipped with a binary operation $* :  X \times X \to X$ and a map $f : X \to X$ satisfying the following conditions:

For each $x\in X$, the identity
\begin{equation} \label{fCon-I}
x* x =f(x)
\end{equation}

holds. For any $x,y \in X$, there exists a unique $z \in X$ such that
\begin{equation} \label{fCon-II}
z* y = f(x).
\end{equation}
\begin{equation} \label{fCon-III}
(x *y) * f(z) =  (x * z) * (y * z)
\end{equation}
\end{defn}
We denote $f$-quandle by $(X, *, f)$.\\

\noindent Any $\mathbb{Z}[\omega^{\pm},\beta]$-module $M$ is an $f$-quandle with 
\begin{equation*} x*y = \omega \cdot  x +\beta \cdot  y   \end{equation*} 
for $x, y \in M $ with $\omega \beta = \beta \omega $, and we call it an \textit{Alexander $f$-quandle} (\cite[Example 2.1 item (4)]{CEGM-2016}).

\begin{rmk}
	When $f$ is the identity map and $\beta = 1 - \omega$ above, then  $(X, *)$ is a quandle and $(M,*)$ is an Alexander quandle as usual.
\end{rmk}



\begin{thm} (\cite[Theorem 5.1]{CEGM-2016}) \label{T2.1} Let $(X, *,f)$ be a $f$-quandle, $f$ be a quandle morphism  and $A$ be an abelian group.The following family of operators $ \delta^n : C^n(X) \to C^{n+1}(X)$ defines a cohomology complex $C^{*}(X, *,f, A).$
\begin{eqnarray*}
\lefteqn{
\delta^n \phi (x_1, \dots, x_{n+1}) } \nonumber \\ && =
(-1)^{n+1} \sum_{i=2}^{n+1} (-1)^{i}    \eta_{[x_1, \dots, \hat{x}_i, \dots, x_{n+1}],f^{\{i-2\}}[x_i, \dots, x_{n+1}]} \phi(x_1, \dots, \hat{x}_i, \dots, x_{n+1})
\nonumber \\
&&
- (-1)^{n+1} \sum_{i=2}^{n+1} (-1)^{i}   \phi  (x_1 \ast x_i, x_2 \ast  x_i, \dots, x_{i-1}\ast  x_i,f( x_{i+1}), \dots, f(x_{n+1}))
\nonumber\\
&&
+ (-1)^{n+1}  \tau_{[x_1, x_3, \dots, x_{n+1}],[x_2, \dots, x_{n+1}]} \phi (x_2, \dots, x_{n+1}),
\end{eqnarray*}
where $  [x_1, x_2, x_3, x_4, \dots, x_n] =  (( \dots (x_1 * x_2) * f(x_3)) * f^2(x_4))* \dots ) *  f^{n-2}(x_n)$. Note that  for $ i <n,$ we have \\

 $ [x_1, x_2, x_3, x_4,\dots, x_n] = [x_1,\dots, \hat{x}_i,\dots, x_n] * f^{i-2}[x_i,\dots, x_n]$

\end{thm}

As in the standard quandle cohomology theory, the degenerate subcomplex is given by $C^D_n = \{ (x_1, x_2, ....., x_n) \in X^n \, ; \, x_i = x_{i+1} \, \,  \text{for} \,  i \geq 2 \}$. A similar degenerate subcomplex appeared in  \cite{NP} under the name of  \textit{late degenerate quandles}.

%
%
%
%
%
%


Under the assumption that $\eta = id$ and $\tau = 0$, we can re-write the cohomology complex in Theorem~\ref{T2.1} as follows.

\begin{equation} \label{E3.3}
\begin{split}
&\delta^n \phi (x_1, \dots, x_{n+1}) \cr
&= (-1)^{n+1} \sum_{i=2}^{n+1} (-1)^{i}     \phi(x_1, \dots, \hat{x}_i, \dots, x_{n+1}) \cr
&- (-1)^{n+1} \sum_{i=2}^{n+1} (-1)^{i}   \phi  (x_1 \ast x_i, x_2 \ast  x_i, \dots, x_{i-1}\ast  x_i,f( x_{i+1}), \dots, f(x_{n+1})).
\end{split}
\end{equation} 

We will reformulate the $f$-quandle cohomology for convenient of calculations.

Let $ U_1=x_1-x_2, U_2=x_2-x_3, \ldots,  U_i = x_i-x_{i+1}, \ldots , U_{n}=x_{n}-x_{n+1}, U_{n+1} = x_{n+1}$ for $i=1,2, \cdots, n.$

Then \eqref{E3.3} becomes

\begin{equation} \label{E3.4}
\begin{split}
&\delta^n \phi (U_1, \dots, U_{n+1})  \cr
&= (-1)^{n+1} \sum_{i=1}^{n} (-1)^{i}     \phi(U_1, \dots, U_{i-1}, U_i+U_{i+1}, U_{i+2}, \dots, U_{n+1}) \cr
&- (-1)^{n+1} \sum_{i=1}^{n} (-1)^{i}   \phi  (\omega\,U_1, \omega\,U_2, \dots, \omega\,U_{i-1}, \omega\,U_i+(\omega+\beta)\,U_{i+1}, f(U_{i+2}), \dots, f(U_{n+1}))
\end{split}
\end{equation}

The following formula is a generalization of \cite[Eq. (3)]{TN-2013} when $\eta = id$ and $\tau = 0$ with \\
$C^n_d(X) := \{ \sum a_{i_1, \cdots,i_n} \cdot U_1^{i_1} \cdots U_n^{i_n} \in C^n(X) \mid \sum_{1\leq k \leq n}i_k=d\}$  and $ degree (f_a) = d_a.$

\begin{equation} \label{E2.1}
\begin{split}
&\delta_n(f)(U_1,\ldots,U_n,U_{n+1})= \displaystyle\sum_{0\leq a\leq p-1}\,\delta_{n-1}(f_a)(U_1,\ldots, U_n)\cdot U_{n+1}^a \cr
& + (-1)^{n-1}\, \displaystyle\sum_{0\leq a\leq p-1}\,f_a(U_1,\ldots, U_{n-1})(U_n+U_{n+1})^a \cr
&- (-1)^{n-1}\, \displaystyle\sum_{0\leq a\leq p-1}\,f_a(U_1,\ldots, U_{n-1}) \, \omega^{d_a}\, (\omega +\beta)^{d - d_a-a}\, (\omega \,U_n+(\omega +\beta)\,U_{n+1})^a.
\end{split}
\end{equation}


\section{The 2-cocycles}\label{2cocy}
In this section, we investigate the $2$-cocycles.  Precisely we provide basis of the second cohomology $H_Q^2((X, *, f); \f_q)$.

\begin{prop}
If $\omega^{p^t+p^s} = 1$ and $(\omega+\beta)^{p^t+p^s} = 1$, where $s$ and $t$ are non-negative integers, then $U_1^{p^t}U_2^{p^s}$ is a $2$-cocycle.
\end{prop}

\begin{proof}
By \eqref{E3.4}, we have 

$$\delta(U_1^{p^t}) = (U_1+U_2)^{p^t} - (\omega\,U_1 + (\omega +\beta)\,U_2)^{p^t}.$$


Then it follows from \eqref{E3.4} and \eqref{E2.1} that 

\begin{equation}\label{N1}
\begin{split}
\delta(U_1^{p^t} U_2^{p^s})&= \delta(U_1^{p^t}) U_3^{p^s}-U_1^{p^t}(U_2+U_3)^{p^s} \cr
&+U_1^{p^t} \omega^{d_a} (\omega+\beta)^{d-d_a - a} (\omega \,U_2 + (\omega+\beta) \,U_3)^{p^s}.
\end{split}
\end{equation}

Also, note that  $  d_a = p^t, a = p^s$ and $ d = p^t+ p^s $. Then we have from \eqref{N1}

\begin{equation}\label{New1}
\begin{split}
\delta(U_1^{p^t} U_2^{p^s})&= (U_1+U_2)^{p^t} U_3^{p^s}- (\omega +\beta)^{p^s} (\omega\,U_1 + (\omega +\beta)\,U_2)^{p^t} U_3^{p^s}-U_1^{p^t}(U_2+U_3)^{p^s} \cr
&+U_1^{p^t} \omega^{p^t} (\omega+\beta)^0  (\omega \,U_2 + (\omega +\beta) \,U_3)^{p^s}\cr
&=(1-\omega^{p^t}(\omega+\beta)^{p^s}) U_1^{p^t}U_3^{p^s}\,+\,(1-(\omega +\beta)^{p^s+p^t}) U_2^{p^t}U_3^{p^s}\cr
&-(1-\omega^{p^t}(\omega+\beta)^{p^s}) U_1^{p^t}U_3^{p^s}\,-\,(1-\omega^{p^s+p^t}) U_2^{p^t}U_3^{p^s}.
\end{split}
 \end{equation}
 
Since $\omega^{p^t+p^s} = 1$ and $(\omega+\beta)^{p^t+p^s} = 1$, the right hand side of \eqref{New1} is $0$. This completes the proof.
\end{proof}
\begin{thm}\label{T1}
Fix $\omega, \beta \in \f_q$ with $\omega \neq 0, 1$. Let $X$ be the corresponding Alexander $f$-quandle on $\f_q$. Then the set 

$$\{U_1^{p^v}U_2^{p^u} \mid \omega^{p^v+p^u} = 1, (\omega+\beta)^{p^v+p^u} = 1; \, \,  0 \leq v < u < m\} $$

provides a basis of the second cohomology $H_Q^2((X, *, f); \f_q)$.

\end{thm}

\begin{exmp}
Let $p$ be an odd prime and $v, u$ be non-negative integers. Let $\omega = -1$ and $ \beta = 2$. Then we have $\omega^{p^v+p^u} = 1$ and $(\omega + \beta)^{p^v+p^u} = 1$. Hence, the set defined in Theorem~\ref{T1} provides a basis for $2$-cocycles. 
\end{exmp}

\begin{exmp}
Let $f(x)=x^2+x+1\in \f_2[x]$ and consider $\f_{4}=\f_2[x]/(f)$. Let $\omega$ be a primitive element of $\f_{4}$. Then the order of $\omega$ is $3$. Let $\beta= \omega^{2}$. Note that $\omega^{2}=\omega +1$ and $\omega^{2}$ is also a primitive element of $\f_{4}$ since it is a conjugate of $\omega$ with respect to $\f_2$. We have
$$\omega^{2^0+2^1}=1\,\,and \,\, (\omega+\beta)^{2^0+2^1}=1,$$

Hence $\{U_1^{2^0}U_2^{2^1}\} $ provides a basis of the second cohomology $H_Q^2((X, *, f); \f_{4})$.

\end{exmp}


\section{The 3-cocycles}\label{3cocy}

In this section we give basis for the cohomology group $H_Q^3((X, *, f); \f_q)$.

For positive integers $a$ and $b$, let 

$$\mu_a(x,y)=(x+y)^a-x^a-y^a$$

and define 

$$\psi(a,b):=\,\,(\mu_a(U_1, U_2)-\mu_a(\omega\,U_1, (\omega+\beta)\,U_2))\cdot U_3^b.$$

Then we have the following proposition. 

\begin{prop}
If $\omega^{a+p^s} = 1$ and $ (\omega+\beta)^{a+p^s} = 1$, then $\Psi (a, p^s)$ is a $3$-cocycle.
\end{prop}

\begin{proof}
Define $$h(U_1, U_2)=\mu_a(U_1, U_2)-\mu_a(\omega\,U_1, (\omega+\beta)\,U_2).$$

Note that 
$$\psi(a,b):=\,\,h(U_1, U_2)\cdot U_3^b.$$

Then by \eqref{E3.4}, we have 

$$\delta(U_1^a) = (U_1+U_2)^a - (\omega\,U_1 + (\omega +\beta) U_2)^a ,$$

which implies 

$$h(U_1, U_2)= \delta(U_1^a) -  (1-\omega^a)\cdot U_1^a - (1-(\omega+\beta)^a)\cdot U_2^a.$$

Also, from \eqref{E3.4}, we have 

\begin{equation}\label{E3.5}
\begin{split}
& \delta(h(U_1, U_2)) \cr 
&= - h(U_1+U_2, U_3) + h(\omega U_1 + (\omega+\beta)U_2, (\omega+\beta)U_3) + h(U_1, U_2+U_3) \cr 
&- h(\omega U_1, \omega U_2 +(\omega+\beta)U_3) \cr
&= (1-\omega^a)\,h(U_1, U_2)-(1-(\omega+\beta)^a)\,h(U_2, U_3)\cr
&= \big( h(U_1, U_2) - h(U_2, U_3)\big) - \big(\omega^a\,h(U_1, U_2) - (\omega+\beta)^a\,h(U_2, U_3)\big).
\end{split}
\end{equation}

Since 

$$\psi(a,b)=\,\,h(U_1, U_2)\cdot U_3^b, $$

from \eqref{E2.1} and  \eqref{E3.5} we have 

\begin{equation} \label{E3.6}
\begin{split}
&\delta(\Psi(a, b)) \cr 
&= \delta(h(U_1, U_2))\cdot U_4^{b} - h(U_1, U_2)\, \delta(U_3^b)\cr 
&= \Big[\big( h(U_1, U_2) - h(U_2, U_3)\big) - (\omega+\beta)^b \big(\omega^a\,h(U_1, U_2) - (\omega+\beta)^a\,h(U_2, U_3)\big)\Big]\,U_4^b \cr 
&- h(U_1, U_2)\, \big((U_3+U_4)^{b} -\omega^a\,(\omega\,U_3 + (\omega+\beta)\,U_4)^{b}\big).
\end{split}
\end{equation}

Let $b=p^s$. Then, from \eqref{E3.6} we have 

\begin{equation} \label{E3.7}
\begin{split}
&\delta(\Psi(a, p^s)) \cr 
&=(1-\omega^a\,(\omega +\beta)^{p^s})\,h(U_1,U_2) \,U_4^{p^s} - (1-(\omega+\beta)^{a+p^s})\, h(U_2, T_3)\,U_4^{p^s} \cr 
&-(1-\omega^{a+p^s})\,h(U_1, U_2)\,U_3^{p^s}-(1-\omega^a\,(\omega +\beta)^{p^s})\,h(U_1, U_2)\,U_4^{p^s}.
\end{split}
\end{equation}

Since $\omega^{a+p^s} = 1$ and $ (\omega+\beta)^{a+p^s} = 1$, the right hand side of \eqref{E3.7} is $0$. This completes the proof. 

\end{proof}

\begin{rmk}
Moreover, $\Psi$, defined above, is a coboundary; see \cite{JMandemaker}.
\end{rmk}

Let $\chi(x, y)= \displaystyle\sum_{i=1}^{p-1}\,(-1)^{i-1}\cdot i^{-1}\cdot x^{p-i}\cdot y^i \equiv \displaystyle\frac{1}{p} ((x+y)^p - x^p - y^p) \pmod{p}$. 

Define 

$$E_0(a \cdot p, b)= \Big(\chi(U_1, U_2)^a - (\omega+\beta)^b\,\chi(\omega\,U_1, (\omega+\beta)\,U_2)^a\Big)\cdot U_3^b.$$

Also, define 

$$h(U_1,U_2):=\chi(U_1,U_2)^a-(\omega+\beta)^b\,\chi(\omega\,U_1, (\omega+\beta)\,U_2)^a.$$

Then we have

$$E_0(a \cdot p, b)= h(U_1, U_2)\cdot U_3^b.$$

Hence we have the following proposition. 

\begin{prop}
If $\omega^{p^s+p^h}=1$ and $(\omega+\beta)^{p^s+p^h}=1$ with $s > 0$, then $E_0(p^s, p^h)$ is a $3-$cocycle.
\end{prop}

\begin{proof}

\begin{equation}\label{E3.1}
\begin{split}
\delta(E_0(a\cdot p, b))&= \delta(h(U_1, U_2))\cdot U_4^b -\,h(U_1, U_2)\,\delta(U_3^b) \cr
&= (1-\omega^{ap}\,(\omega +\beta)^b)\,h(U_1, U_2)\,U_4^b\,-\,(1-(\omega +\beta)^{ap+b})\,h(U_2, U_3)\,U_4^b \cr
&-\,h(U_1, U_2)\,\big((U_3+U_4)^b-\omega^{ap}\,(\omega\,U_3+(\omega+\beta)\,U_4)^b\big). \cr
\end{split}
\end{equation}

Let $a=p^{s-1}$ and $b=p^h$. Then from equation \eqref{E3.1} we have 

\begin{equation}\label{E3.2}
\begin{split}
\delta(E_0(p^s, p^h))&= (1-\omega^{p^{s}}\,(\omega +\beta)^{p^h})\,h(U_1,U_2)\cdot U_4^{p^h} \cr 
&-\,(1-(\omega+\beta)^{p^{s}+p^{h}})\,h(U_2,U_3)\cdot U_4^{p^h} \cr        
&-\,(1-\omega^{p^s+p^h})\,U_3^{p^h}\,h(U_1,U_2)\cr
&-\,(1-\omega^{p^{s}}\,(\omega +\beta)^{p^h})\,h(U_1,U_2)\cdot U_4^{p^h}.
\end{split}
\end{equation}

Since $\omega^{p^s+p^h}=1$ and $(\omega+\beta)^{p^s+p^h}=1$, the right hand side of \eqref{E3.2} is $0$. This completes the proof. 

\end{proof}

Again, let $\chi(x, y)= \displaystyle\sum_{i=1}^{p-1}\,(-1)^{i-1}\cdot i^{-1}\cdot x^{p-i}\cdot y^i \equiv \displaystyle\frac{1}{p} ((x+y)^p - x^p - y^p) \pmod{p}$. 

Define 

$$E_1(a, b\cdot p)= U_1^a \cdot \Big(\chi(U_2,U_3)^b-\omega^{a}\,\chi(\omega\,U_2,\, (\omega+\beta)\,U_3)^b\Big).$$

Also, define 

$$h(U_2,U_3):=\chi(U_2,U_3)^b-\omega^{a}\,\chi(\omega\,U_2,\, (\omega+\beta)\,U_3)^b.$$

Then we have the following propositon. 

\begin{prop}
If $\omega^{p^s+p^t}=1$ and $(\omega+\beta)^{p^s+p^t}=1$ with $s > 0$, then $E_1(p^t, p^s)$ is a $3-$cocycle.
\end{prop}

\begin{proof}

Note that 

$$E_1(a, b\cdot p)= U_1^a \cdot h(U_2,U_3).$$ 

We have

\begin{equation}\label{E4.8}
\begin{split}
&\delta(E_1(a, b\cdot p))\cr
&=\delta(U_1^a \cdot h(U_2,U_3)) \cr
&=\delta(U_1^a)\,h(U_3, U_4)-U_1^a\,\delta(h(U_2, U_3))\cr
&=\big((U_1+U_2)^a-(\omega +\beta)^{p\cdot b}\,(\omega\,U_1+(\omega +\beta)\,U_2)^a\big)\,h(U_3, U_4)\cr
&-U_1^a\,\big(h(U_2+U_3, U_4)-\omega^a\,h(\omega\,U_2+(\omega +\beta)\,U_3,\,(\omega +\beta)\,U_4)\big)\cr
&+U_1^a\,\big(h(U_2, U_3+U_4)-\omega^a\,h(\omega\,U_2, \omega\,U_3+\,(\omega +\beta)\,U_4)\big).
\end{split}
\end{equation}

Let $a=p^t$ and $b=p^{s-1}$. Then from \eqref{E4.8} we have

\begin{equation}\label{E4.9}
\begin{split}
&\delta(E_1(p^t, p^s))\cr
&=\,U_1^{p^t}\cdot \Big[(1-\omega^{p^t}(\omega +\beta)^{p^s})\, h(U_3, U_4)- h(U_2+U_3,\,U_4) \cr 
&+ \omega^{p^t}\, h(\omega \,U_2+(\omega+\beta)\,U_3,\,(\omega +\beta)\,U_4)+ h(U_2, U_3+U_4)\cr 
&-\omega^{p^t}\, h(\omega \,U_2, \omega\,U_3+(\omega+\beta)\,U_4)\Big] +\,(1-(\omega +\beta)^{p^t+p^s})\,U_2^{p^t}\,h(U_3, U_4).
\end{split}
\end{equation}

Since $h(U_i,U_{i+1})=\chi(U_i,U_{i+1})^{p^{s-1}}-\omega^{p^t}\,\chi(\omega\,U_i,\, (\omega+\beta)\,U_{i+1})^{p^{s-1}}$,  $\omega^{p^s+p^h}=1$,  and $(\omega+\beta)^{p^s+p^t}=1$, straightforward computation yields that the right hand side of \eqref{E4.9} is $0$. This completes the proof. 

\end{proof}

Let $p$ be a prime, and $v$, $u$, and $t$ be non-negative integers.
Define $F(p^v, p^u,p^ t) = U_1^{p^v} U_2^{p^u} U_3^{p^t} \in C^3$ where $p^v, p^u, p^ t < q$. 

\begin{prop}
\begin{enumerate}
\item  If $\omega^{p^v+ p^u+p^t}=1$ and $(\omega+\beta)^{p^v+ p^u+p^t}=1$, then $F(p^v, p^u, p^t)$ is a $3$-cocycle. 
\item If $\omega^{p^v+ p^u}=1$ and $(\omega+\beta)^{p^v+ p^u}=1$, then $F(p^v, p^u, 0)$ is a $3$-cocycle.
\end{enumerate}
\end{prop}

\begin{proof}

We first prove (1). 

\begin{equation}\label{NN1}
\begin{split}
&\delta(F(p^v, p^u, p^t))\cr
& = \delta(U_1^{p^v} U_2^{p^u} U_3^{p^t}) \cr
& =\, \big((U_1+U_2)^{p^v} - (\omega+\beta)^{p^u+p^t}(\omega U_1+ (\omega+\beta)U_2)^{p^v}\big) \cdot U_3^{p^u} \cdot U_4^{p^t}  \cr
& -\, U_1^{p^v} \cdot \big((U_2+U_3)^{p^u} -  \omega^{p^v} (\omega+\beta)^{p^t} (\omega U_2+ (\omega+\beta)U_3)^{p^u}\big) \cdot U_4^{p^t}\cr
&+\,  U_1^{p^v} \cdot U_2^{p^u} \big((U_3+U_4)^{p^t} -  \omega^{p^v+p^u}  (\omega U_3+ (\omega+\beta)\,U_4)^{p^t}\big) \cr
&=\, (1-\omega^{p^v}(\omega+\beta)^{p^u+p^t}) U_1^{p^v} U_3^{p^u} U_4^{p^t} + (1-(\omega+\beta)^{p^v+p^u+p^t}) U_2^{p^v} U_3^{p^u} U_4^{p^t} \cr
&-\, ( 1 - \omega^{p^v+p^u}(\omega+\beta)^{p^t})U_1^{p^v} U_2^{p^u} U_4^{p^t} - (1 - \omega^{p^v}(\omega+\beta)^{p^t+p^u})U_1^{p^v} U_3^{p^u} U_4^{p^t} \cr
&+\,  (1 - \omega^{p^v+p^u+p^t})U_1^{p^v} U_2^{p^u} U_3^{p^t} +  (1 - \omega^{p^v+p^u}(\omega+\beta)^{p^t})U_1^{p^v} U_2^{p^u} U_4^{p^t}\cr
&= \, 0.
\end{split}
\end{equation}

\noindent Since $\omega^{p^v+ p^u+p^t}=1$ and $(\omega+\beta)^{p^v+ p^u+p^t}=1$, the right hand side of \eqref{NN1} is $0$. 

\noindent In (2), by taking $p^t$ as $0$ in \eqref{NN1}, and with $\omega^{p^v+ p^u}=1$ and $(\omega+\beta)^{p^v+ p^u}=1$, it can be shown in a similar manner that $$\delta(F(p^v, p^u, 0)) =  0.$$
\end{proof}

\noindent As in \cite{JMandemaker, TM-2005}, let $Q$ be the set of all tuples $(p^v, p^u, p^t, p^s)$ where $p$ is a prime, such that $v < t, u < s, u \leq t$ and $ \omega^{p^v+p^t} = \omega^{p^u+p^s} = (\omega+\beta)^{p^v+p^t} = (\omega+\beta)^{p^u+p^s} = 1$, and one of the following conditions hold. \\
Case I.\,\,\,\,\,\,\,\,$\omega^{p^v+p^u} = 1, (\omega+\beta)^{p^v+p^u} = 1.$\\
Case II.\,\,\,\,\,\,$\omega^{p^v+p^u} \neq 1,(\omega+\beta)^{p^v+p^u} \neq 1 $ and $ t > s.$\\
Case III.\,\,\,\,$\omega^{p^v+p^u} \neq 1, (\omega+\beta)^{p^v+p^u} \neq 1, t = s$ and $p \neq 2.$\\
Case IV.\,\,\,\,$\omega^{p^v+p^u} \neq 1, (\omega+\beta)^{p^v+p^u} \neq 1, u \leq v < t < s$ and $\omega^{p^v} = \omega^{p^u}$, $(\omega+\beta)^{p^v} = (\omega+\beta)^{p^u}$ when $p \neq 2.$\\
Case V.\,\,\,\,\,\,$\omega^{p^v+p^u} \neq 1,(\omega+\beta)^{p^v+p^u} \neq 1,  u < v < t \leq s$ and $\omega^{p^v} = \omega^{p^u}$, $(\omega+\beta)^{p^v} = (\omega+\beta)^{p^u}$ when $p = 2.$

\noindent Moreover, if $p=2$, we need $ u < t$ as well.

\noindent For each $(p^v, p^u, p^t, p^s) \in Q$ , we denote a cocycle by $\Gamma$.\\ 

\noindent Then we have the following proposition discussing case I.\,\,$\omega^{p^v+p^u} = 1, (\omega+\beta)^{p^v+p^u} = 1.$

\begin{prop} \label{CaseI}
$\Gamma(p^v, p^u, p^t, p^s) = F(p^v, p^u+p^t, p^s)$ is a $3$-cocycle.
\end{prop}

\begin{proof} 
\begin{equation}
\begin{split}
&\delta(F(p^v, p^u+p^t, p^s))\cr
&=\, \delta(U_1^{p^v}U_2^{p^u+p^t}U_3^{p^s})\cr
&=\,  \delta(U_1^{p^v})U_3^{p^u+p^t}U_4^{p^s} -  U_1^{p^v}\delta(U_3^{p^u+p^t})U_4^{p^s}  +  U_1^{p^v}U_3^{p^u+p^t} \delta(U_4^{p^s}) \cr
&=\, ((U_1+U_2)^{p^v} - (\omega+\beta)^{p^u+p^t+p^s} (\omega U_1+ (\omega+\beta)U_2)^{p^v} U_3^{p^u+p^t}U_4^{p^s}\cr
&-\, U_1^{p^v} \cdot ( (U_2+U_3)^{p^u+p^t} - \omega^{p^v} (\omega+\beta)^{p^s}(\omega+\beta)^{p^v} (\omega U_2+(\omega+\beta)U_3)^{p^u+p^t}))U_4^{p^s}\cr
&+\, U_1^{p^v} \cdot U_2^{p^u+p^t} \cdot ( (U_3+U_4)^{p^s} - \omega^{p^v+p^u+p^t} (\omega+\beta)^{p^v+p^u+p^t} (\omega U_3+(\omega+\beta)U_4)^{p^s}) 
\end{split}
\end{equation}
Note that $(x+y)^{p^u+p^t} = (x^{p^u}+y^{p^u})(x^{p^t}+y^{p^t})$, this reduced to
\begin{equation}
\begin{split}
&\delta(U_1^{p^v}U_2^{p^u+p^t}U_3^{p^s})\cr
&=\, (1- \omega^{p^v} (\omega+\beta)^{p^u+p^t+p^s})U_1^{p^v}U_3^{p^u+p^t}U_4^{p^s} +  (1- (\omega+\beta)^{p^u+p^t+p^s+p^v}) U_2^{p^v}U_3^{p^u+p^t}U_4^{p^s}\cr
&-\, (1-\omega^{p^v+p^u+p^t}(\omega+\beta)^{p^s} )U_1^{p^v}U_2^{p^u+p^t}U_4^{p^s} - (1- \omega^{p^v}(\omega+\beta)^{p^s+p^u+p^t} ) U_1^{p^v}U_3^{p^u+p^t}U_4^{p^s}\cr
&-\, (1- \omega^{p^v+p^u}(\omega+\beta)^{p^s+p^t} )U_1^{p^v}U_2^{p^u}U_3^{p^t}U_4^{p^s} -  (1- \omega^{p^v+p^t} (\omega+\beta)^{p^s+p^u})U_1^{p^v}U_2^{p^t}U_3^{p^u}U_4^{p^s}\cr
&+\, (1- \omega^{p^v+p^u+p^t+p^s} )U_1^{p^v}U_2^{p^u+p^t}U_3^{p^s} + (1- \omega^{p^v+p^u+p^t}(\omega+\beta)^{p^s} )U_1^{p^v}U_2^{p^u+p^t}U_4^{p^s}\cr
&=\, 0.
\end{split}
\end{equation}
\end{proof}

\noindent Then we have the following proposition discussing case II, $\omega^{p^v+p^u} \neq 1,(\omega+\beta)^{p^v+p^u} \neq 1 $ and $ t > s.$
\begin{prop} 
$ \Gamma(p^v, p^u, p^t, p^s) = F(p^v, p^u+p^t, p^s) - F(p^u, p^v+p^s, p^t) - (\omega^{p^u}(\omega+\beta)^{p^s}-1)^{-1} (1 - \omega^{p^u+p^v}(\omega+\beta)^{p^t+p^s}) F(p^v, p^u, p^t+p^s) + F(p^v+p^u, p^s, p^t)$  is a $3$-cocycle.
\end{prop}

\begin{proof} 
\begin{equation}
\begin{split}
&\delta(F(p^v, p^u+p^t, p^s)) - \delta(F(p^u, p^v+p^s, p^t)) \cr
&-\, (\omega^{p^u}(\omega+\beta)^{p^s}-1)^{-1} (1 - \omega^{p^u+p^v}(\omega+\beta)^{p^t+p^s}) \delta(F(p^v, p^u, p^t+p^s) + \delta(F(p^v+p^u, p^s, p^t)))\cr
&=\, -( 1 - \omega^{p^v+p^u}(\omega+\beta)^{p^s+p^t}) U_1^{p^v}U_2^{p^u}U_3^{p^t}U_4^{p^s} + ( 1 - \omega^{p^v+p^u}(\omega+\beta)^{p^s+p^t}) U_1^{p^u}U_2^{p^v}U_3^{p^s}U_4^{p^t}\cr
&-\, (\omega^{p^u}(\omega+\beta)^{p^s}-1)^{-1} (1 - \omega^{p^u+p^v}(\omega+\beta)^{p^t+p^s}) [ (1 - \omega^{p^v+p^u+p^t}(\omega+\beta)^{p^s}) U_1^{p^v}U_2^{p^u}U_3^{p^t}U_4^{p^s}\cr
&-\,  (1 - \omega^{p^u}(\omega+\beta)^{p^s+p^t+p^v}) U_1^{p^u}U_2^{p^v}U_3^{p^s}U_4^{p^t}]\cr
&=\, 0.
\end{split}
\end{equation}
\end{proof}

\noindent Then we have the following proposition discussing case III, $\omega^{p^v+p^u} \neq 1, (\omega+\beta)^{p^v+p^u} \neq 1, t = s$ and $p \neq 2.$ In \cite{TN-2014}, it is shown we can present this case as follows:
\begin{prop} 
$\Gamma(p^v, p^u, p^t, p^s) = F(p^v, p^t+ p^s, p^u)$ is a $3$-cocycle.
\end{prop}
\begin{proof}
The proof is similar to that of Proposition ~\ref{CaseI}.
\end{proof}
\noindent Then we have the following proposition discussing cases IV and V, $\omega^{p^v+p^u} \neq 1, (\omega+\beta)^{p^v+p^u} \neq 1, u \leq v < t < s$ and $\omega^{p^v} = \omega^{p^u}$, $(\omega+\beta)^{p^v} = (\omega+\beta)^{p^u}$ when $p \neq 2$ and $\omega^{p^v+p^u} \neq 1,(\omega+\beta)^{p^v+p^u} \neq 1,  u < v < t \leq s$ and $\omega^{p^v} = \omega^{p^u}$, $(\omega+\beta)^{p^v} = (\omega+\beta)^{p^u}$ when $p = 2.$In \cite{TN-2014}, it is shown we can present this case as follows:
\begin{prop} 

$\Gamma(p^v, p^u, p^t, p^s) = F(p^t, p^v+p^u, p^s)$ is a $3$-cocycle.

\end{prop}
\begin{proof} 
The proof is similar to that of Proposition ~\ref{CaseI}.
\end{proof}

\begin{thm}
Fix $\omega, \beta \in \f_q$ with $\omega \neq 0, \pm 1$. Let $X$ be the corresponding Alexander $f$-quandle on $\f_q$ where  $H_Q^2((X, *, f); \f_q)  \cong 0$. Then the set  \\
\begin{equation*}
\begin{split}
& I = \{F(p^v, p^u, p^t) \mid \omega^{p^v+p^u+p^t} = (\omega+\beta)^{p^v+p^u+p^t} = 1, p^v < p^u < p^t <q\}\cr
&\cup \,  \{F(p^v, p^u, 0) \mid \omega^{p^v+p^u} = (\omega+\beta)^{p^v+p^u} = 1, p^v < p^u  <q\}\cr
&\cup  \, \{E_0(p \cdot p^v, p^u) \mid \omega^{p^{v+1}+p^u} = (\omega+\beta)^{p^{v+1}+p^u} = 1, p^v < p^u < q\}\cr
&\cup \,  \{E_1(p^v, p \cdot p^u) \mid  \omega^{ p^v+ p^{u+1}} = (\omega+\beta)^{p^v+ p^{u+1}} = 1, p^v \leq p^u < q\}\cr
&\cup \,  \{\Gamma(p^v, p^u, p^t, p^s) \mid (p^v, p^u, p^t, p^s) \in Q(q)\}
\end{split}
\end{equation*}
provides a basis of the third cohomology $H_Q^3((X, *, f); \f_q)$.

\end{thm}

\begin{exmp}\label{EE}
Let $p$ be an odd prime and $v, u$ and $ t$ be non-negative integers. Let $\omega = -1$ and $\beta = 2$. Hence, $\omega^{p^v+p^u+p^t} \neq 1$ and $(\omega+\beta)^{p^v+p^u+p^t} = 1$. The we have the following. 

\begin{enumerate}
\item $F(p^v, p^u, p^t)$ is not a $3$-cocycle. 
\item $F(p^v, p^u, 0)$ is a $3$-cocycle since $\omega^{p^v+p^u} = 1$ and $(\omega+\beta)^{p^v+p^u} = 1$. Also, $E_0(p^{v+1},p^u)$ and $E_1(p^v, p^{u+1})$ are $3$-cocycles.

\noindent Moreover, 

\noindent $Q(q) = \{ (p^v, p^u,p^t, p^s) \mid p^u \leq p^t, p^v < p^t, p^u < p^s\}$, and $\omega^{p^v+p^u} = (\omega+\beta)^{p^v+p^u} = 1$ for any $(p^v, p^u, p^t, p^s) \in Q(q)$.

\vskip 0.1in

\noindent Therefore, 

\[
\begin{split}
&\{F(p^v, p^u, 0) \mid 0<p^v<p^u<q\} \cup \{E_0(p^{v+1}, p^u)\mid p^v<p^u<q\} \cup \{E_1(p^v, p^{u+1}) \mid p^v < p^u<q\}\cr
&\cup \{F(p^v, p^u+p^t, p^s)\mid p^u \leq p^t, p^v<p^t, p^u<p^s, p^i<q,\, \textnormal{for all}\,i \in \{v, u, t,s\}\}
\end{split}
\]

is a basis for the cohomology group $H_Q^3((X, *, f); \f_q)$.

\end{enumerate} 

\end{exmp}

\begin{rmk}
Example~\ref{EE} shows that when $\beta=1-\omega$, the basis for the cohomology group $H_Q^3((X, *, f); \f_q)$ above is the same as the basis for the cohomology group $H_Q^3((X, *); \f_q)$, explained in \cite[Subsection~2.4.1]{TM-2005}. 
\end{rmk}

\begin{exmp}
Let $f(x)=x^3+x^2+1\in \f_2[x]$ and consider $\f_{8}=\f_2[x]/(f)$. Let $\omega$ be a primitive element of $\f_{8}$. Then the order of $\omega$ is $7$. Let $\beta= \omega^{2^2}$. Note that $\omega^{2^2}=\omega^3 +\omega$ and $\omega^{2^2}$ is also a primitive element of $\f_{8}$ since it is a conjugate of $\omega$ with respect to $\f_2$. We have
$$\omega^{2^0+2^1+2^2}=1\,\,and \,\, (\omega+\beta)^{2^0+2^1+2^2}=1,$$

but $\omega^{2^i+2^j}\neq 1$ for $i, j\in \{0,1,2\}$. Hence $H_Q^3((X, *, f); \f_{8})$ is generated by $\{F(2^0, 2^1, 2^2)$\}. 
\end{exmp}

\section{The 4-cocycles}\label{4cocy}

In this section, we give some propositions showing some particular polynomials are $4$-cocycles.  The main theorem gives basis for the cohomology group $H_Q^4((X, *, f); \f_q)$ under the condition that the group $H_Q^2((X, *, f); \f_q)$ is trivial.

\begin{prop}
If \,$\omega^{p^v+p^u+p^t+p^s}=1$ and $(\omega+\beta)^{p^v+p^u+p^t+p^s}=1$, then the polynomial $U_1^{p^v}U_2^{p^u}U_3^{p^t}U_4^{p^s}$ is a $4$-cocycle.
\end{prop}

\begin{proof}
\begin{equation}\label{NN2}
\begin{split}
&\delta(U_1^{p^v}U_2^{p^u}U_3^{p^t}U_4^{p^s})\cr
&=\, ((U_1+U_2)^{p^v}-  (\omega+\beta)^{p^u+p^t+p^s}\,(\omega U_1+ (\omega+\beta)U_2)^{p^v}) U_3^{p^u}U_4^{p^t}U_5^{p^s}\cr
&-\, U_1^{p^v} \cdot ( (U_2+U_3)^{p^u} - \omega^{p^v}\,(\omega+\beta)^{p^t+p^s}\,(\omega U_2+(\omega+\beta)U_3)^{p^u})\,U_4^{p^t}U_5^{p^s}\cr
&+\, U_1^{p^v} \cdot U_2^{p^u} \cdot ( (U_3+U_4)^{p^t} - \omega^{p^v+p^u} (\omega+\beta)^{p^s} (\omega U_3+(\omega+\beta)U_4)^{p^t})\,U_5^{p^s} \cr
&- \ U_1^{p^v}  \cdot U_2^{p^u}  \cdot U_3^{p^t}  \cdot ( (U_4+U_5)^{p^s} - \omega^{p^v+p^u+p^t}\,(\omega U_4+(\omega+\beta)U_5)^{p^s})\cr
&=\, (1-\omega^{p^v}\, (\omega+\beta)^{p^u+p^t+p^s})\,U_1^{p^v}U_3^{p^u}U_4^{p^t}U_5^{p^s} + (1-(\omega+\beta)^{p^v+p^u+p^t+p^s})U_2^{p^v}U_3^{p^u}U_4^{p^t}U_5^{p^s}\cr
&-\, (1 - \omega^{p^v+p^u}(\omega+\beta)^{p^t+p^s})U_1^{p^v}U_2^{p^u}U_4^{p^t}U_5^{p^s} - (1 - \omega^{p^v}(\omega+\beta)^{p^u+p^t+p^s})U_1^{p^v}U_3^{p^u}U_4^{p^t}U_5^{p^s}\cr
&+\, (1 - \omega^{p^v+p^u+p^t}(\omega+\beta)^{p^s})\,U_1^{p^v}U_2^{p^u}U_3^{p^t}U_5^{p^s} + (1 - \omega^{p^v+p^u}(\omega+\beta)^{p^s+p^t})U_1^{p^v}U_2^{p^u}U_4^{p^t}U_5^{p^s}\cr
&-\, (1 - \omega^{p^v+p^u+p^t+p^s})U_1^{p^v}U_2^{p^u}U_3^{p^t}U_4^{p^s} -\,  (1 - \omega^{p^v+p^u+p^t} (\omega+\beta)^{p^s})U_1^{p^v}U_2^{p^u}U_3^{p^t}U_5^{p^s}.
\end{split}
\end{equation}

Since $\omega^{p^v+p^u+p^t+p^s}=1$ and $(\omega+\beta)^{p^v+p^u+p^t+p^s}=1$, the right hand side of \eqref{NN2} is $0$. This completes the proof. 

\end{proof}

We recall $\chi(x, y)= \displaystyle\sum_{i=1}^{p-1}\,(-1)^{i-1}\cdot i^{-1}\cdot x^{p-i}\cdot y^i \equiv \displaystyle\frac{1}{p} ((x+y)^p - x^p - y^p) \pmod{p}$. 

\begin{prop} 
If $\omega^{p^{u+1}+p^t+p^s} = 1$ and $(\omega+\beta)^{p^{u+1}+p^t+p^s}=1$, then the polynomial $\Big(\chi(U_1, U_2)^{p^u} - (\omega+\beta)^{p^t+p^s}\chi(\omega\,U_1, (\omega+\beta)\,U_2)^{p^u}\Big)U_3^{p^t}U_4^{p^s}$ is a $4$-cocycle.
\end{prop}

\begin{proof}

Let $h(U_1, U_2)=\,\chi(U_1, U_2)^{p^u} - (\omega+\beta)^{p^t+p^s}\chi(\omega\,U_1, (\omega+\beta)\,U_2)^{p^u}$. 

We show that $h(U_1, U_2)\,U_3^{p^t}\,U_4^{p^s}$ is a $4$-cocycle. 

\begin{equation}\label{NN7}
\begin{split}
&\delta(h(U_1, U_2)\,U_3^{p^t}U_4^{p^s}) \cr
&= \delta(h(U_1, U_2)\,U_4^{p^t}U_5^{p^s} \cr
&- h(U_1, U_2)\, \big( (U_3+U_4)^{p^t}\,-\,\omega^{p^{u+1}}\,(\omega +\beta)^{p^{s}}\,(\omega\,U_3+(\omega +\beta)\,U_4)^{p^t}\big)\,U_5^{p^s} \cr
&+ h(U_1, U_2)\,U_3^{p^t}\,\big( (U_4+U_5)^{p^s}\,-\,\omega^{p^{u+1}+p^t}\,(\omega\,U_4+(\omega +\beta)\,U_5)^{p^s}\big)\cr
&= (1-\omega^{p^{u+1}}\,(\omega +\beta)^{p^s+p^t})\,h(U_1, U_2)\,U_4^{p^t}\,U_5^{p^s}-(1-(\omega +\beta)^{p^{u+1}+p^t+p^s})\,h(U_2, U_3)\,U_4^{p^t}\,U_5^{p^s} \cr
&- (1-\omega^{p^{u+1}+p^t}\,(\omega +\beta)^{p^{s}})\,h(U_1, U_2)\,U_3^{p^t}\,U_5^{p^s}-(1-\omega^{p^{u+1}}\,(\omega +\beta)^{p^{s}+p^t})\,h(U_1, U_2)\,U_4^{p^t}\,U_5^{p^s} \cr
&+ (1-\omega^{p^{u+1}+p^t+p^s})\,h(U_1, U_2)\,U_3^{p^t}\,U_4^{p^s} + (1-\omega^{p^{u+1}+p^t}\,(\omega +\beta)^{p^s})\,h(U_1, U_2)\,U_3^{p^t}\,U_5^{p^s}.
\end{split}
\end{equation}

Since $\omega^{p^{u+1}+p^t+p^s} = 1$ and $(\omega+\beta)^{p^{u+1}+p^t+p^s}=1$, the right hand side of \eqref{NN7} is $0$. This completes the proof. 

\end{proof}

\begin{prop} 
If $\omega^{p^v+p^{t+1}+p^s} = 1$ and $(\omega+\beta)^{p^v+p^{t+1}+p^s}=1$, then the polynomial $ U_1^{p^v} \Big(\chi(U_2, U_3)^{p^t} -  \omega^{p^v} (\omega+\beta)^{p^s}\chi(\omega U_2, (\omega+\beta)U_3)^{p^t} \Big) U_4^{p^s}$ is a $4$-cocycle.
\end{prop}

\begin{proof}

Let $h(U_2, U_3)=\,\chi(U_2, U_3)^{p^t} -  \omega^{p^v} (\omega+\beta)^{p^s}\chi(\omega U_2, (\omega+\beta)U_3)^{p^t}$. 

Now we show that $U_1^{p^v}\,h(U_2, U_3)\,U_4^{p^s}$ is a $4$-cocycle. 

\begin{equation}\label{NN8}
\begin{split}
&\delta(U_1^{p^v}\,h(U_2, U_3)\,U_4^{p^s}) \cr
&= \big((U_1+U_2)^{p^v}-\,(\omega +\beta)^{p^{t+1}+p^s}\,(\omega\,U_1+\,(\omega +\beta)\,U_2)^{p^v}\big)\,h(U_3, U_4)\,U_5^{p^s} \cr
&- U_1^{p^v}\, \big( h(U_2+U_3, U_4)\,-\,\omega^{p^{v}}\,(\omega +\beta)^{p^{s}}\,h(\omega\,U_2+(\omega +\beta)\,U_3, (\omega +\beta)\,U_4)\big)\,U_5^{p^s} \cr
&+  U_1^{p^v}\,\big( h(U_2, U_3+U_4)\,-\,\omega^{p^{v}}\,(\omega +\beta)^{p^{s}}\,h(\omega\,U_2, \omega\,U_3+ (\omega +\beta)\,U_4)\big)\,U_5^{p^s} \cr
&- U_1^{p^v}\,h(U_2, U_3)\,\big((U_4+U_5)^{p^s}-\,\omega^{p^v+p^{t+1}}\,(\omega\,U_4+\,(\omega +\beta)\,U_5)^{p^s}\big)\cr
&= U_1^{p^v}\,\Big[(1-\omega^{p^v}\,(\omega +\beta)^{p^{t+1}+p^s})\,h(U_3, U_4)-\,h(U_2+U_3, U_4)+ h(U_2,\,U_3+U_4)\cr
&+\,\omega^{p^v}\,(\omega +\beta)^{p^s}\,h(\omega\,U_2+\,(\omega +\beta)\,U_3, (\omega +\beta)\,U_4) \cr
&-\,\omega^{p^v}\,(\omega +\beta)^{p^s}\,h(\omega\,U_2,\,\omega\,U_3+\,(\omega +\beta)\,U_4)-\,(1-\omega^{p^v+p^{t+1}}\,(\omega +\beta)^{p^s})\,h(U_2, U_3)\Big]\,U_5^{p^s}\cr
&+ (1-(\omega +\beta)^{p^v+p^{t+1}+p^s})\,U_2^{p^v}\,h(U_3, U_4)\, U_5^{p^s}-\,(1-\omega^{p^v+p^{t+1}+p^s})\,U_1^{p^v}\,h(U_2, U_3)\,U_4^{p^s}.
\end{split}
\end{equation}

Since 
$$h(U_i, U_{i+1})=\,\chi(U_i, U_{i+1})^{p^t} -  \omega^{p^v} (\omega+\beta)^{p^s}\chi(\omega U_i, (\omega+\beta)\,U_{i+1})^{p^t},$$ 

\noindent and $\omega^{p^v+p^{t+1}+p^s} = 1$ and $(\omega+\beta)^{p^v+p^{t+1}+p^s}=1$, it can be shown that the right hand side of \eqref{NN8} is $0$. This completes the proof. 

\end{proof}

\begin{prop} 
If $\omega^{p^v+p^u+p^{s+1}} = 1$ and $(\omega+\beta)^{p^v+p^u+p^{s+1}}=1$, then the polynomial 
$$U_1^{p^v} U_2^{p^u} \Big(\chi(U_3, U_4)^{p^s} - \omega^{p^v+p^u}\,\chi(\omega\,U_3, (\omega+\beta)\,U_4)^{p^s}\Big)$$ is a $4$-cocycle.
\end{prop}

\begin{proof}

Let $h(U_3, U_4)= \chi(U_3, U_4)^{p^s} - \omega^{p^v+p^u}\,\chi(\omega\,U_3, (\omega+\beta)\,U_4)^{p^s}$. 

Now we claim that $U_1^{p^v} U_2^{p^u} h(U_3, U_4)$ is a $4$-cocycle.

\begin{equation}\label{D1}
\begin{split}
&\delta(U_1^{p^v} U_2^{p^u} h(U_3, U_4))\cr
&=\delta(U_1^{p^v}) U_3^{p^u} h(U_4, U_5)-U_1^{p^v} \delta(U_2^{p^u}) h(U_4, U_5))+U_1^{p^v} U_2^{p^u} \delta(h(U_3, U_4))\cr
&=\big((U_1+U_2)^{p^v}-(\omega+\beta)^{p^u+p^{s+1}}\,(\omega\,U_1+(\omega +\beta)\,U_2)^{p^v}\big)U_3^{p^u} h(U_4, U_5)\cr
&-U_1^{p^v}\,\big((U_2+U_3)^{p^u}-\omega^{p^v}(\omega+\beta)^{p^{s+1}}\,(\omega\,U_2+(\omega +\beta)\,U_3)^{p^u}\big) h(U_4, U_5)\cr
&+U_1^{p^v}U_2^{p^u}\,\Big(h(U_3+U_4, U_5)-\omega^{p^v+p^u}\,h(\omega\,U_3+(\omega+\beta)\,U_4, (\omega+\beta)\,U_5)\Big)\cr
&-U_1^{p^v}U_2^{p^u}\,\Big(h(U_3, U_4+U_5)-\omega^{p^v+p^u}\,h(\omega\,U_3, \omega\,U_4+(\omega+\beta)\,U_5)\Big)\cr
&=(1-(\omega+\beta)^{p^v+p^u+p^{s+1}})\,U_2^{p^v}U_3^{p^u}h(U_4, U_5)+U_1^{p^v}U_2^{p^u}\,\Big[h(U_3+U_4, U_5)-h(U_3, U_4+U_5)\cr
&-(1-\omega^{p^v+p^u}\,(\omega+\beta)^{p^{s+1}})\,h(U_4, U_5)+\omega^{p^v+p^u}\,h(\omega\,U_3, \omega\,U_4+(\omega+\beta)\,U_5)\cr
&-\omega^{p^v+p^u}\,h(\omega\,U_3+(\omega+\beta)\,U_4, (\omega+\beta)\,U_5)\Big]
\end{split}
\end{equation}

Since $$h(U_i, U_{i+1})= \chi(U_i, U_{i+1})^{p^s} - \omega^{p^v+p^u}\,\chi(\omega\,U_i, (\omega+\beta)\,U_{i+1})^{p^s},$$ 

$\omega^{p^v+p^u+p^{s+1}} = 1$, and $(\omega+\beta)^{p^v+p^u+p^{s+1}}=1$, it can be shown that the right hand side of \eqref{D1} is $0$. This completes the proof. 

\end{proof}

\begin{prop} 
If $\omega^{p^i+p^j+p^u+p^t+p^s}=1$, $(\omega+\beta)^{p^i+p^j+p^u+p^t+p^s}=1$, $\omega^{p^i+p^j}=\omega^{p^i+p^u}=1$ and $(\omega+\beta)^{p^u+p^t+p^s}=(\omega+\beta)^{p^j+p^t+p^s}=1$, then the polynomial $U_1^{p^i}U_2^{p^j+p^u}U_3^{p^t}U_4^{p^s}$ is a $4$-cocycle.
\end{prop}

\begin{proof} 

\begin{equation}\label{NN3}
\begin{split}
&\delta(U_1^{p^i}U_2^{p^j+p^u}U_3^{p^t}U_4^{p^s})\cr
&=\, ((U_1+U_2)^{p^i}-(\omega +\beta)^{p^j+p^u+p^t+p^s}\,(\omega U_1+ (\omega+\beta)U_2)^{p^i}) U_3^{p^j+p^u}U_4^{p^t}U_5^{p^s}\cr
&-\, U_1^{p^i} \cdot ( (U_2+U_3)^{p^j+p^u} - \omega^{p^i} (\omega+\beta)^{p^t+p^s} (\omega U_2+(\omega+\beta)U_3)^{p^j+p^u}))U_4^{p^t}U_5^{p^s}\cr
&+\, U_1^{p^i} \cdot U_2^{p^j+p^u} \cdot ( (U_3+U_4)^{p^t} - \omega^{p^i+p^j+p^u} (\omega+\beta)^{p^s} (\omega U_3+(\omega+\beta)U_4)^{p^t}))U_5^{p^s} \cr
&- \ U_1^{p^i}  \cdot U_2^{p^j+p^u}  \cdot U_3^{p^t}  \cdot ( (U_4+U_5)^{p^s} - \omega^{p^i+p^j+p^u+p^t}\,(\omega U_4+(\omega+\beta)U_5)^{p^s})\cr
\end{split}
\end{equation}

Note that $(x+y)^{p^j+p^u} = (x^{p^j}+y^{p^j})(x^{p^u}+y^{p^u})$. Hence, from \eqref{NN3} we have

\begin{equation}\label{NN4}
\begin{split}
&\delta(U_1^{p^i}U_2^{p^j+p^u}U_3^{p^t}U_4^{p^s})\cr
&=\, (1- \omega^{p^i}\,(\omega +\beta)^{p^j+p^u+p^t+p^s})U_1^{p^i}U_3^{p^j+p^u}U_4^{p^t}U_5^{p^s}\cr
&+(1- (\omega +\beta)^{p^i+p^j+p^u+p^t+p^s}) U_2^{p^i}U_3^{p^j+p^u}U_4^{p^t}U_5^{p^s}\cr
&-(1-\omega^{p^i+p^j+p^u} (\omega+\beta)^{p^t+p^s})U_1^{p^i}U_2^{p^j+p^u}U_4^{p^t}U_5^{p^s}\cr
&- (1- \omega^{p^i} (\omega+\beta)^{p^j+p^u+p^t+p^s})\, U_1^{p^i}U_3^{p^j+p^u}U_4^{p^t}U_5^{p^s}\cr
&-\, (1- \omega^{p^i+p^j} (\omega+\beta)^{p^u+p^t+p^s})U_1^{p^i}U_2^{p^j}U_3^{p^u}U_4^{p^t}U_5^{p^s} \cr
&- (1- \omega^{p^i+p^u} (\omega+\beta)^{p^j+p^t+p^s})U_1^{p^i}U_2^{p^u}U_3^{p^j}U_4^{p^t}U_5^{p^s}\cr
&+\, (1- \omega^{p^i+p^j+p^u+p^t} (\omega+\beta)^{p^s})U_1^{p^i}U_2^{p^j+p^u}U_3^{p^t}U_5^{p^s} \cr
&+\,  (1- \omega^{p^i+p^j+p^u} (\omega+\beta)^{p^s+p^t})U_1^{p^i}U_2^{p^j+p^u}U_4^{p^t}U_5^{p^s}\cr
&-\, (1- \omega^{p^i+p^j+p^u+p^t+p^s})\,U_1^{p^i}U_2^{p^j+p^u}U_3^{p^t}U_4^{p^s}\cr
&-\, (1- \omega^{p^i+p^j+p^u+p^t} (\omega+\beta)^{p^s})\,U_1^{p^i}U_2^{p^j+p^u}U_3^{p^t}U_5^{p^s}.
\end{split}
\end{equation}

Since $\omega^{p^i+p^j+p^u+p^t+p^s}=1$, $(\omega+\beta)^{p^i+p^j+p^u+p^t+p^s}=1$, $\omega^{p^i+p^j}=\omega^{p^i+p^u}=1$ and $(\omega+\beta)^{p^u+p^t+p^s}=(\omega+\beta)^{p^j+p^t+p^s}=1$, the right hand side of \eqref{NN4} is $0$. This completes the proof. 

\end{proof}

\begin{prop} 
If $\omega^{p^{u+1}+p^{s+1}} = 1$ and $(\omega+\beta)^{p^{u+1}+p^{s+1}} =1$, then the polynomial 
$$\Big(\chi(U_1, U_2)^{p^u} - (\omega +\beta)^{p^{s+1}}\,\chi(\omega\,U_1, (\omega +\beta)\,U_2)^{p^u}\Big) \Big(\chi(U_3, U_4)^{p^s} - \omega^{p^{u+1}}\,\chi(\omega\,U_3, (\omega +\beta)\,U_4)^{p^s}\Big)$$ is a $4$-cocycle.
\end{prop}

\begin{proof}

Let 

$$h(U_1, U_2)=\chi(U_1, U_2)^{p^u} - (\omega +\beta)^{p^{s+1}}\,\chi(\omega\,U_1, (\omega +\beta)\,U_2)^{p^u}$$ 

and 

$$h^*(U_3, U_4)=\chi(U_3, U_4)^{p^s} - \omega^{p^{u+1}}\,\chi(\omega\,U_3, (\omega +\beta)\,U_4)^{p^s}.$$

We claim that $h(U_1, U_2)\,h^*(U_3, U_4)$ is a $4$-cocycle.

\begin{equation}\label{D3}
\begin{split}
&\delta(h(U_1, U_2)\,h^*(U_3, U_4))\cr
&=\delta(h(U_1, U_2))\,h^*(U_4, U_5)-h(U_1, U_2)\,\delta(h^*(U_3, U_4))\cr
&=h(U_1,U_2)\Big[h^*(U_3, U_4+U_5)-h^*(U_3+U_4, U_5)+(1-\omega^{p^{u+1}}(\omega+\beta)^{p^{s+1}})h^*(U_4, U_5)\cr
&-\omega^{p^{u+1}}\,h^*(\omega\,U_3, \omega\,U_4+(\omega+\beta)\,U_5)+\omega^{p^{u+1}}\,h^*(\omega\,U_3+(\omega+\beta)\,U_4, (\omega+\beta)\,U_5)\Big]\cr
&-(1-(\omega+\beta)^{p^{u+1}+p^{s+1}})\,h(U_2, U_3)h^*(U_4, U_5)
\end{split}
\end{equation}

Since $h(U_i, U_{i+1})=\chi(U_i, U_{i+1})^{p^u} - (\omega +\beta)^{p^{s+1}}\,\chi(\omega\,U_i, (\omega +\beta)\,U_{i+1})^{p^u}$, 

\noindent $h^*(U_i, U_{i+1})=\chi(U_i, U_{i+1})^{p^s} - \omega^{p^{u+1}}\,\chi(\omega\,U_i, (\omega +\beta)\,U_{i+1})^{p^s}$, $\omega^{p^{u+1}+p^{s+1}} = 1$, and

\noindent $(\omega+\beta)^{p^{u+1}+p^{s+1}} =1$, it can be shown that the right hand side of \eqref{D3} is $0$. This completes the proof. 

\end{proof}

\begin{prop} 
If $\omega^{p^i+p^j+p^u+p^{s+1}} = 1$, $(\omega+\beta)^{p^i+p^j+p^u+p^{s+1}} = 1$, $\omega^{p^i+p^j}=\omega^{p^i+p^u}=1$, and $(\omega+\beta)^{p^u+p^{s+1}} = (\omega+\beta)^{p^j+p^{s+1}}=1$, then the polynomial 

$$U_1^{p^i}  U_2^{p^j+p^u}  \Big(\chi(U_3, U_4)^{p^s} - \omega^{p^i+p^j+p^u}\,\chi(\omega\,U_3, (\omega+\beta)\,U_4)^{p^s}\Big)$$ is a $4$-cocycle.

\end{prop}

\begin{proof}
Let $h(U_3, U_4)=\chi(U_3, U_4)^{p^s} - \omega^{p^i+p^j+p^u}\,\chi(\omega\,U_3, (\omega+\beta)\,U_4)^{p^s}$. We claim that $U_1^{p^i}\,U_2^{p^j+p^u}\,h(U_3, U_4)$ is a $4$-cocycle. 

\begin{equation}\label{R1}
\begin{split}
&\delta(U_1^{p^i}\,U_2^{p^j+p^u}\,h(U_3, U_4))\cr
&=\delta(U_1^{p^i})\,U_3^{p^j+p^u}\,h(U_4, U_5)-U_1^{p^i}\,\delta(U_2^{p^j+p^u})\,h(U_4, U_5)+U_1^{p^i}\,U_2^{p^j+p^u}\,\delta(h(U_3, U_4))\cr
&=((U_1+U_2)^{p^i} - (\omega+\beta)^{p^j+p^u+p^{s+1}}\,(\omega\,U_1+ (\omega+\beta)U_2)^{p^i})\,U_3^{p^j+p^u}\,h(U_4, U_5)\cr
&-\, U_1^{p^i} \cdot ( (U_2+U_3)^{p^j+p^u} - \omega^{p^i} (\omega+\beta)^{p^{s+1}} (\omega\,U_2+(\omega+\beta)\,U_3)^{p^j+p^u})\,h(U_4, U_5)\cr
&+U_1^{p^i}\,U_2^{p^j+p^u}\,\big(h(U_3+U_4, U_5)-\omega^{p^i+p^j+p^u}\,h(\omega\,U_3+(\omega+\beta)\,U_4, (\omega+\beta)\,U_5)\big)\cr
&-U_1^{p^i}\,U_2^{p^j+p^u}\,\big(h(U_3, U_4+U_5)-\omega^{p^i+p^j+p^u}\,h(\omega\,U_3, \omega\,U_4+(\omega+\beta)\,U_5)\big) \cr
&=(1-\omega^{p^i}\,(\omega+\beta)^{p^j+p^u+p^{s+1}})\,U_1^{p^i}\,U_3^{p^j+p^u}\,h(U_4, U_5)\cr
&+(1-(\omega+\beta)^{p^i+p^j+p^u+p^{s+1}})\,U_2^{p^i}\,U_3^{p^j+p^u}\,h(U_4, U_5)\cr
&-(1-\omega^{p^i+p^j+p^u}\,(\omega+\beta)^{p^{s+1}})\,U_1^{p^i}\,U_2^{p^j+p^u}\,h(U_4, U_5)\cr
&-(1-\omega^{p^i+p^j}\,(\omega+\beta)^{p^u+p^{s+1}})\,U_1^{p^i}\,U_2^{p^j}\,U_3^{p^u}\,h(U_4, U_5)\cr
&- (1-\omega^{p^i+p^u}\,(\omega+\beta)^{p^j+p^{s+1}})\,U_1^{p^i}\,U_2^{p^u}\,U_3^{p^j}\,h(U_4, U_5)\cr
&- (1-\omega^{p^i}\,(\omega+\beta)^{p^j+p^u+p^{s+1}})\,U_1^{p^i}\,U_3^{p^j+p^u}\,h(U_4, U_5)\cr
&+U_1^{p^i}\,U_2^{p^j+p^u}\,\big(h(U_3+U_4, U_5)-\omega^{p^i+p^j+p^u}\,h(\omega\,U_3+(\omega+\beta)\,U_4, (\omega+\beta)\,U_5)\big)\cr
&-U_1^{p^i}\,U_2^{p^j+p^u}\,\big(h(U_3, U_4+U_5)-\omega^{p^i+p^j+p^u}\,h(\omega\,U_3, \omega\,U_4+(\omega+\beta)\,U_5)\big)
\end{split}
\end{equation}

\noindent Since $h(U_i, U_{i+1})=\chi(U_i, U_{i+1})^{p^s} - \omega^{p^i+p^j+p^u}\,\chi(\omega\,U_i, (\omega+\beta)\,U_{i+1})^{p^s}$, $\omega^{p^i+p^j+p^u+p^{s+1}} = 1$, $(\omega+\beta)^{p^i+p^j+p^u+p^{s+1}} = 1$, $\omega^{p^i+p^j}=\omega^{p^i+p^u}=1$, and $(\omega+\beta)^{p^u+p^{s+1}} = (\omega+\beta)^{p^j+p^{s+1}}=1$, the right hand side of \eqref{R1} is $0$. This completes the proof.

\end{proof}

\begin{prop} 
If $\omega^{p^i+p^j+p^v+p^u+p^t+p^s}=1$, $(\omega+\beta)^{p^i+p^j+p^v+p^u+p^t+p^s}=1$, $\omega^{p^i+p^j} = \omega^{p^i+p^v} = \omega^{p^v+p^u} = \omega^{p^v+p^t} =1$, and $(\omega +\beta)^{p^s+p^t} =(\omega +\beta)^{p^s+p^u} = (\omega +\beta)^{p^v+p^u} = (\omega +\beta)^{p^j+p^u} =1$, then the polynomial $U_1^{p^i}U_2^{p^j+p^v}U_3^{p^u+p^t}U_4^{p^s}$ is a $4$-cocycle.
\end{prop}

\begin{proof} 

\begin{equation}\label{NN5}
\begin{split}
&\delta(U_1^{p^i}U_2^{p^j+p^v}U_3^{p^u+p^t}U_4^{p^s})\cr
&=\, ((U_1+U_2)^{p^i} - (\omega+\beta)^{p^j+p^v+p^u+p^t+p^s}\,(\omega\,U_1+ (\omega+\beta)U_2)^{p^i})\, U_3^{p^j+p^v}U_4^{p^u+p^t}U_5^{p^s}\cr
&-\, U_1^{p^i} \cdot ( (U_2+U_3)^{p^j+p^v} - \omega^{p^i} (\omega+\beta)^{p^u+p^t+p^s} (\omega\,U_2+(\omega+\beta)U_3)^{p^j+p^v})U_4^{p^u+p^t}U_5^{p^s}\cr
&+\, U_1^{p^i} \cdot U_2^{p^j+p^v} \cdot ( (U_3+U_4)^{p^u+p^t} - \omega^{p^i+p^j+p^v} (\omega+\beta)^{p^s} (\omega U_3+(\omega+\beta)U_4)^{p^u+p^t})\,U_5^{p^s} \cr
&- \ U_1^{p^i}  \cdot U_2^{p^j+p^v}  \cdot U_3^{p^u+p^t}  \cdot ( (U_4+U_5)^{p^s} - \omega^{p^i+p^j+p^v+p^u+p^t}\,(\omega\,U_4+(\omega+\beta)U_5)^{p^s}).
\end{split}
\end{equation}

Note that $(x+y)^{p^j+p^u} = (x^{p^j}+y^{p^j})(x^{p^u}+y^{p^u})$. Hence, from \eqref{NN5} we have

\begin{equation}\label{NN6}
\begin{split}
&\delta(U_1^{p^i}U_2^{p^j+p^v}U_3^{p^u+p^t}U_4^{p^s})\cr
&=\, (1- \omega^{p^i}(\omega+\beta)^{p^j+p^v+p^u+p^t+p^s})U_1^{p^i}U_3^{p^j+p^v}U_4^{p^u+p^t}U_5^{p^s} \cr
&+ (1- (\omega+\beta)^{p^i+p^j+p^v+p^u+p^t+p^s})\,U_2^{p^i}U_3^{p^j+p^v}U_4^{p^u+p^t}U_5^{p^s}\cr
&-\, (1-\omega^{p^i+p^j+p^v} (\omega+\beta)^{p^u+p^t+p^s})U_1^{p^i}U_2^{p^j+p^v}U_4^{p^u+p^t}U_5^{p^s}\cr
&- (1- \omega^{p^i} (\omega+\beta)^{p^j+p^v+p^u+p^t+p^s}) U_1^{p^i}U_3^{p^j+p^v}U_4^{p^u+p^t}U_5^{p^s}\cr
&-\, (1- \omega^{p^i+p^j} (\omega+\beta)^{p^v+p^u+p^t+p^s})U_1^{p^i}U_2^{p^j}U_3^{p^v}U_4^{p^u+p^t}U_5^{p^s}\cr
&- (1- \omega^{p^i+p^v} (\omega+\beta)^{p^j+p^u+p^t+p^s})U_1^{p^i}U_2^{p^v}U_3^{p^j}U_4^{p^u+p^t}U_5^{p^s}\cr
&+\, (1- \omega^{p^i+p^j+p^v+p^u+p^t} (\omega+\beta)^{p^s})\,U_1^{p^i}U_2^{p^j+p^v}U_3^{p^u+p^t}U_5^{p^s} \cr
&+\,  (1- \omega^{p^i+p^j+p^v} (\omega+\beta)^{p^u+p^t+p^s})\,U_1^{p^i}U_2^{p^j+p^v}U_4^{p^u+p^t}U_5^{p^s}\cr
&+\, (1- \omega^{p^i+p^j+p^v+p^u} (\omega+\beta)^{p^s+p^t})\,U_1^{p^i}U_2^{p^j+p^v}U_3^{p^u}U_4^{p^t}U_5^{p^s}\cr
&+\, (1- \omega^{p^i+p^j+p^v+p^t} (\omega+\beta)^{p^s+p^u})\,U_1^{p^i}U_2^{p^j+p^v}U_3^{p^t}U_4^{p^u}U_5^{p^s}\cr
&-\, (1- \omega^{p^i+p^j+p^v+p^u+p^t+p^s})\,U_1^{p^i}U_2^{p^j+p^v}U_3^{p^u+p^t}U_4^{p^s}\cr
&-\, (1- \omega^{p^i+p^j+p^v+p^u+p^t} (\omega+\beta)^{p^s})\,U_1^{p^i}U_2^{p^j+p^v}U_3^{p^u+p^t}U_5^{p^s}.
\end{split}
\end{equation}

Since$\omega^{p^i+p^j+p^v+p^u+p^t+p^s}=1$, $(\omega+\beta)^{p^i+p^j+p^v+p^u+p^t+p^s}=1$, $\omega^{p^i+p^j} = \omega^{p^i+p^v} = \omega^{p^v+p^u} = \omega^{p^v+p^t} =1$, and $(\omega +\beta)^{p^s+p^t} =(\omega +\beta)^{p^s+p^u} = (\omega +\beta)^{p^v+p^u} = (\omega +\beta)^{p^j+p^u} =1$, the right hand side of \eqref{NN6} is $0$. This completes the proof. 

\end{proof}

\noindent Let  $q=p^m$, where $m$ is a positive integer. 

\vskip 0.1in

\noindent $A =\Big\{ U_1^{p^v}U_2^{p^u}U_3^{p^t}U_4^{p^s}\mid \omega^{p^v+p^u+p^t+p^s}=1, (\omega+\beta)^{p^v+p^u+p^t+p^s}=1, 0\leq v<u<t<s<m \Big\}$,

\vskip 0.1in

\noindent $B = \Big\{ \big(\chi(U_1, U_2)^{p^u} - (\omega+\beta)^{p^t+p^s}\chi(\omega\,U_1, (\omega+\beta)\,U_2)^{p^u}\big)U_3^{p^t}U_4^{p^s}\mid \omega^{p^{u+1}+p^t+p^s} = 1, (\omega+\beta)^{p^{u+1}+p^t+p^s}=1, 0 \leq u<t<s<m \Big\} $,

\vskip 0.1in

\noindent $C = \Big\{ U_1^{p^v} \big(\chi(U_2, U_3)^{p^t} -  \omega^{p^v} (\omega+\beta)^{p^s}\chi(\omega U_2, (\omega+\beta)U_3)^{p^t} \big) U_4^{p^s} \mid \omega^{p^v+p^{t+1}+p^s} = 1, (\omega+\beta)^{p^v+p^{t+1}+p^s}=1, 0\leq v \leq t<s<m \Big\} $,

\vskip 0.1in

\noindent $D = \Big\{ U_1^{p^v} U_2^{p^u} \big(\chi(U_3, U_4)^{p^s} - \omega^{p^v+p^u}\,\chi(\omega\,U_3, (\omega+\beta)\,U_4)^{p^s}\big) \mid \omega^{p^v+p^u+p^{s+1}} = 1, (\omega+\beta)^{p^v+p^u+p^{s+1}}=1, 0 \leq v < u \leq s < m \Big\}$, and

\vskip 0.1in

\noindent $E = \Gamma(p^v, p^u, p^t, 0) = \Big\{ U_1^{p^v}U_2^{p^u}U_3^{p^t} \mid \omega^{p^v+p^u+p^t}=1, (\omega+\beta)^{p^v+p^u+p^t}=1, 0 \leq v < u<t<m \Big\} $.

\vskip 0.1in

Then we have the following theorem. 

\begin{thm}
Fix $\omega, \beta \in \f_q$ with $\omega \neq 0, \pm 1$. Let $X$ be the corresponding Alexander $f$-quandle on $\f_q$ where  $H_Q^2((X, *, f); \f_q)  \cong 0$. Then the set  
$ A \cup B \cup C \cup D \cup E$ 
provides a basis of the fourth cohomology $H_Q^4((X, *, f); \f_q)$.
\end{thm}

\begin{exmp}
Let $f(x)=x^4+x+1\in \f_2[x]$ and consider $\f_{16}=\f_2[x]/(f)$. Let $\omega$ be a primitive element of $\f_{16}$. Then the order of $\omega$ is $15$. Let $\beta= \omega^{2^2}$. Note that $\omega^{2^2}=\omega +1$ and $\omega^{2^2}$ is also a primitive element of $\f_{16}$ since it is a conjugate of $\omega$ with respect to $\f_2$. We have
$$\omega^{2^0+2^1+2^2+2^3}=1\,\,and \,\, (\omega+\beta)^{2^0+2^1+2^2+2^3}=1,$$

but $\omega^{2^i+2^j+2^k}\neq 1$ for $i, j, k\in \{0,1,2,3\}$. Hence $H_Q^4((X, *, f); \f_{16})$ is generated by $A$. 
\end{exmp}

\noindent \textbf{Acknowledgements}: 
The authors would like to thank Takefumi Nosaka and Masahico Saito for their valuable comments and helpful suggestions which improved the presentation of the paper.


\end{document}